\documentclass[10pt,twocolumn]{IEEEtran}

\IEEEoverridecommandlockouts                              
                                                          \usepackage{amsmath} 
\usepackage{amssymb}  
\usepackage{graphics} 
\usepackage{breqn}
\usepackage{tikz}
\usetikzlibrary{arrows}
\usepackage{bbm}
\usepackage{bbold}
\usepackage{amsmath}
\usepackage{mathtools}
\usepackage{amsfonts}
\usepackage{csquotes}
\usepackage{algorithm}
\usepackage[noend]{algpseudocode}
\usepackage{hyperref}

\makeatletter
\def\BState{\State\hskip-\ALG@thistlm}
\makeatother

\renewcommand\IEEEkeywordsname{Index Terms}

\newtheorem{definition}{Definition}[section]

\newtheorem{remark}{Remark}[section]
\newtheorem{lemma}{Lemma}[section]
\newtheorem{theorem}{Theorem}[section]
\newtheorem{assumption}{Assumption}[section]
\def\QED{~\rule[-1pt]{5pt}{5pt}\par\medskip}
\newenvironment{proof}{{\it  Proof: \ }}{ \hfill \QED}

\title{\LARGE \bf
Distributed Constrained Optimization over Networked Systems via A Singular Perturbation Method 
}

\author{$\textrm{Phuong Huu Hoang}^\dagger$ and $\textrm{Hyo-Sung Ahn}^\dagger$
\thanks{$^\dagger$P. H. Hoang and H.-S. Ahn are with School of Mechanical Engineering,
      Gwangju Institute of Science and Technology, 500-712, Gwangju, Republic of Korea.
      Emails: {\tt\small 
      \{phuonghoang}{\tt\small,hyosung\}@gist.ac.kr}}
}

\begin{document}

\maketitle
\thispagestyle{empty}
\pagestyle{empty}

\begin{abstract}
This paper studies a constrained optimization problem over networked systems with an undirected and connected communication topology. The algorithm proposed in this work utilizes singular perturbation, dynamic average consensus, and saddle point dynamics methods to tackle the problem for a general class of objective function and affine constraints in a fully distributed manner. It is shown that the private information of agents in the interconnected network is guaranteed in our proposed strategy. The theoretical guarantees on the optimality of the solution are provided by rigorous analyses. We apply the new proposed solution into energy networks by a demonstration of two simulations. 
\end{abstract}
\section{Introduction}
In recent years, distributed interconnected multi-agent systems have drawn a large amount of researchers' attention due to their prospects. Of the problems arising in the distributed networked systems, distributed constrained optimization problem (DCOP), in which the feasible solutions are confined to a certain region, appears in various network decision tasks, including optimal resource allocation problem (ORAP) \cite{Kia2017}, \cite{Yi2016}, \cite{Madan2006}-\cite{Doan2017}, economic dispatch problem (EDP) in power grids \cite{Loia2014}-\cite{Yi2016}, and robot motion planning (RMP) in robotic networks \cite{Bullo2009}, \cite{Zhu2015}. The goal of the distributed constrained optimization problem over the networked systems is to seek the optimal values for each agent in a distributed way such that the overall cost of operation of the systems is minimized while respecting constraints.

Distributed constrained optimization and consensus problem in multi-agent networks has been extensively studied recently, see \cite{Nedic2010}, \cite{Sun2017}-\cite{Wang2017} and references therein. In the optimization problem investigated in the aforementioned works, the agents' task is to cooperatively minimize the total objective cost while honoring constraints and reaching a consensus for all agents.  In \cite{Nedic2010}, Nedi\'{c} $et\textrm{ }al.$ tackle the problem in a discontinuous-time fashion based on a distributed projected subgradient method in which the private information of agents in the interconnected network is not guaranteed. The privacy-guaranteed property is also not kept in the proposed strategies presented in \cite{Sun2017}-\cite{Yang2017}.
Our studied problem, which can be found in ORAP, EDP, and RMP, is slightly different that each agent has its own state in the minimizing process. By formulating the problem as our way, the privacy-guaranteed property can be kept during the interaction in the networked systems.
In the literature, optimization problems are solved by numerous discrete-time paradigms \cite{Boyd2004}, \cite{Burke2008}. Thanks to the  well-developed continuous-time stability theory and various powerful mathematical tools, the continuous-time optimization has seen a resurgence of interest in recent years \cite{Ahn2016}-\cite{Yi2016}, \cite{Ye2017b}-\cite{Kia2015}, \cite{Yang2017}-\cite{Wang2017}. 
The idea of using continuous-time saddle point dynamics to find the optimal value of constrained optimization problem has been recently proposed in the works \cite{Durr2011}-\cite{Durr2012}. In these works, a dynamical system is constructed to seek saddle points of Lagrange functions yielding the optimal values. The works, however, consider the problem in a centralized way without equality constraints. Note that in distributed approaches, since each agent only knows its own and its neighbors' information, it is considered more challenging.

As mentioned previously, applications of our studied DCOP are ORAP \cite{Kia2017}, \cite{Yi2016}, \cite{Madan2006}-\cite{Doan2017} and EDP \cite{Loia2014}-\cite{Yi2016}. In \cite{Kia2017}, \cite{Doan2017}, inequality constraints are not of their interests, which is unrealistic in practice, while the objective functions are confined to the quadratic function class in \cite{Loia2014}-\cite{Garcia2012}. In addition, the work \cite{Loia2014} does not provide rigorous theoretical guarantees on the optimality of the solution. There are works investigating more general classes of objective function such as strictly convex functions in \cite{Cherukuri2016}-\cite{Yi2016}.
To tackle the EDP, Cherukuri and  Cort\'{e}s \cite{Cherukuri2016} propose an algorithm in which they modify the original problem into an equivalent one and assume that there exists one node which can obtain the total load capacity of the power network. Moreover, the coincidence of solutions to the original problem and modified one is guaranteed  by  examining a parameter, namely $\epsilon$, which is computed by obtaining information of all agents. The strategy presented in \cite{Yun2017} also encounters a similar drawback in solving the EDP that the convergence depends on a parameter, namely $k$, which is global information. Such aforementioned flaws in \cite{Cherukuri2016}-\cite{Yun2017} make the authors' algorithms seem to be ostensibly distributed. In \cite{Yi2016}, a projection-based method is proposed and discussed to tackle the ORAP. The method is non-smooth and may not be beneficial in terms of computation, projections on complicated sets for instance \cite{SLee2016}.

In this work, we propose a new fully distributed solution to solve the widely-applied DCOP with affine constraints. We consider to tackle the problem in a distributed and smooth manner with privacy-guaranteed property utilizing the saddle point dynamics idea \cite{Durr2011}-\cite{Durr2012}, singular perturbation method \cite{Tan2006}, and dynamic average consensus algorithm \cite{Freeman2006}. Moreover, we take advantage of multi-time-scale property of the singular perturbation method \cite{Khalil2002}, \cite{Tan2006} to design our algorithm.
In essence, our work that can be considered to be a generalization of ORAP and EDP studies strictly convex cost functions with equality and inequality constraints. By confining the objective functions to strictly convex ones, we can obtain the uniqueness of optimal solutions.
Additionally, non-local stability is provided by the proposed algorithm. We rely upon the well-developed singular perturbation theory to provide rigorous theoretical guarantees of our proposed strategy.

The remainder of this paper is organized as follows. In Section II, we provide some notations used throughout this work and for the sake of completeness we briefly present graph theory and dynamic average consensus algorithm used in this paper. We formulate the DCOP in Section III with some assumptions and provide some standard results in convex analysis. Section IV is dedicated to present our distributed solutions along with analyses. We demonstrate the correctness of the new approach by numerical simulations in energy networks in Section V. Section VI ends this paper with conclusions and future directions.
\section{Preliminaries}
We first present notations and basic concepts. 
Let ${\mathbb{R}}$, ${\mathbb{R}}_{\geq 0}$, ${\mathbb{R}}_{> 0}$, $\mathbb{N}$, and $\mathbb{Z}_{\geq 1}$ correspondingly denote the real, non-negative real, positive real, non-negative integer, and positive integer numbers. In addition, $\mathbb{R}_{> \mathbf{v}}^n$ denotes the set of all vectors in $\mathbb{R}^n$ with components greater than the respective components of $\mathbf{v}\in \mathbb{R}^n$. 
Let $\mathbf{I}_n$ represent for the $n \times n$ identity matrix. A matrix $\mathbf{A} \in \mathbb{R}^{n \times n}$ is denoted to be negative definite by $\mathbf{A}\prec 0$ (resp. semi-negative definite $\mathbf{A}\preceq 0$). Let $\otimes$ represent the Kronecker product operator and the superscript $\top$  denotes a transpose of a matrix or a vector. Let $\mathbf{1}_n=[1,...,1]^\top \in \mathbb{R}^n$, while $\mathbf{0}$ represents for all-zero-entry vectors with an appropriate dimension. Given a vector $\mathbf{x}=[x_1,...,x_n]^\top$, $\overline{x}$ represents for $\frac{1}{n}\sum_{i=1}^nx_i$. Additionally, $\frac{\partial f}{\partial x}(y)$ means the derivative of function $f$ with respect to $x$, and then replacing the variable $x$ by $y$. 

We continue to present some basic graph theory \cite{Bullo2009}. A graph is a triplet ${\cal G}=({\cal V}, {\cal E},\mathbf{A})$, where ${\cal V}=\{1,...,n\}$ is the node set and ${\cal E} \subseteq {\cal V}\times {\cal V}$ is the edge set. The adjacency matrix $\mathbf{A}=[a_{ij}] \in \mathbb{R}^{n\times n}$ is defined as $a_{ij}=1$ if node $j$ is connected to node $i$, else, $a_{ij}=0$. The graph is undirected if for every $(i,j) \in {\cal E}$, $(j,i) \in {\cal E}$. An undirected graph is connected if there exists a path between any pair of distinct vertices. The neighboring set of agent $i$ is defined as ${\cal N}_i=\{ j \in {\cal V}|(j,i) \in {\cal E}\}$. The Laplacian matrix $\mathbf{L}=[l_{ij}]$ for the graph is defined as $\mathbf{L}=\mathbf{D}-\mathbf{A}$, where $\mathbf{D}$ is a diagonal matrix whose $i-$th diagonal element is equal to $\sum_{j=1}^na_{ij}$.

We then briefly review the dynamic average consensus algorithm presented in \cite{Freeman2006}.
Let ${\cal G}$ be an undirected and connected graph and $\mathbf{L}$ be its Laplacian matrix. Then, for any constant $\mathbf{u}\in \mathbb{R}^n$, the state of the following system:
\begin{equation*}
\begin{bmatrix}
\dot{\boldsymbol{\xi}}\\
\dot{\boldsymbol{\zeta}}
\end{bmatrix}
=
\begin{bmatrix}
-\mathbf{I}_n-\mathbf{L}& -\mathbf{L}\\
\mathbf{L} &\mathbf{0}
\end{bmatrix} \begin{bmatrix}
\boldsymbol{\xi}\\
\boldsymbol{\zeta}
\end{bmatrix}
+
\begin{bmatrix}
\mathbf{u}\\
\mathbf{0}
\end{bmatrix},
\end{equation*}
with arbitrary initial conditions $\boldsymbol{\xi}(0), \boldsymbol{\zeta}(0) \in \mathbb{R}^n$ remains bounded and $\boldsymbol{\xi}(t)$ converges exponentially to $\frac{1}{n}\mathbf{1}_n^\top \mathbf{u} \mathbf{1}_n$ as $t \rightarrow \infty$.\\
\section{Problem Formulation}
We consider a set of $n \in \mathbb{Z}_{\geq 1}$ agents communicating over an undirected and connected graph $\mathcal{G}=(\mathcal{V},\mathcal{E},\mathbf{A})$. Each agent is represented by a corresponding vertex in the graph. Let $x_i\in \mathbb{R}$ be the state of agent $i$. The objective function of agent $i$ is measured by $f_i(x_i): \mathbb{R}\rightarrow \mathbb{R}$ assumed to be strictly convex and continuously differentiable. Each agent can measure only its own objective function values and the derivative values of the function. The states of agents in the networked system are confined by $l \in \mathbb{N}$ equality constraints $h_e(\mathbf{x})=0$, $e \in {\cal H}=\{1,...,l\}$. In addition, agent $i$ has $m_i \in \mathbb{N}$ local inequality constraints $g_{i{j_i}}({x}_i) \leq 0, j_i \in {\mathcal{G}}_i=\{1,...,m_i\}$. If the set $\cal H$ is empty, i.e., ${\mathcal{H}}=\emptyset$, then there is no equality constraint in the networked system. Similarly, $\mathcal{G}_i=\emptyset$ means agent $i$ does not have any inequality constraint. In this work, we assume that the constraints are affine such that $h_e(\mathbf{x})=\sum_{i=1}^n(a^h_{ie}x_i+b^h_{ie})$ and $g_{ij_i}(x_i)=a_{ij_i}^gx_i+b_{ij_i}^g$, where $a_{ie}^h, b_{ie}^h, a_{ij_i}^g$, and $b_{ij_i}^g\in \mathbb{R}$. The considered problem can be viewed as a generalization of the EDP \cite{Loia2014}-\cite{Yi2016}, in which its supply-demand balance is an equality constraint and limit capacity constraints are inequality ones, and the ORAP \cite{Kia2017}, \cite{Yi2016}, \cite{Doan2017}. Compared to \cite{Kia2017} and \cite{Doan2017}, our work considers inequality constraints which is more challenging. The agents aim to cooperatively minimize the total cost $\sum_{i=1}^nf_i(x_i)$ while respecting the constraints. Let  $\mathbf{x}=[{x}_1,...,{x}_n]^\top\in \mathbb{R}^{n}$. We mathematically state the DCOP as
\begin{subequations}\label{GeneralProblem}
\begin{align}
&\textrm{minimize } f(\mathbf{x})=\sum_{i=1}^n  f_i({x}_i),\\
\text{s.t. }
             &h_e(\mathbf{x})=0, \textrm{ } e \in {\mathcal{H}},\\            
             &g_{i{j_i}}({x}_i) \leq 0,\textrm{ } i \in {\mathcal{V}} \textrm{ and }  j_i \in {\mathcal{G}_i}.
\end{align}
\end{subequations}
We define the associated Lagrangian of the optimization problem (\ref{GeneralProblem}) as
\begin{equation}\label{Lagrangian}
\begin{split}
\mathfrak{L}(\mathbf{x},\boldsymbol{\mu}, \boldsymbol{\lambda})=&f(\mathbf{x})+\sum_{e=1}^l\mu_eh_e(\mathbf{x})\\
&+\sum_{i=1}^n \sum_{j_i=1}^{m_i} \lambda_{ij_i}g_{ij_i}(x_i),
\end{split}
\end{equation}
where $\mu_e \in \mathbb{R}$, $e\in {\cal H}$,  $\lambda_{ij_i}\in \mathbb{R}_{\geq 0}$, $i \in {\cal V}$ and $j_i \in {\cal G}_i$, are Lagrange multipliers and $\boldsymbol{\mu}=[\mu_1,...,\mu_l]^\top \in \mathbb{R}^l$ and $\boldsymbol{\lambda}=[\lambda_{11},...,\lambda_{ij_i},...,\lambda_{nm_n}]^\top \in \mathbb{R}^{\sum_{i=1}^nm_i}_{\geq 0}$. It is well-known that the Lagrange multiplier $\lambda_{ij_i}$ is non-negative \cite{Durr2011}-\cite{Durr2012}, \cite{Boyd2004}, \cite{Burke2008}.\\
We next present some standard results from convexity and optimization in the literature \cite{Durr2011}-\cite{Burke2008}. Let $\mathbf{x}, \mathbf{y} \in \mathbb{R}^n$ and $f: \mathbb{R}^n\rightarrow \mathbb{R}$. The following statements are equivalent for $f\in C^2$:\\
\indent 1) $f$ is convex.\\
\indent 2) $\nabla f(\mathbf{x})^\top(\mathbf{y}-\mathbf{x})\leq f(\mathbf{y})-f(\mathbf{x})$.\\
\indent 3) $\nabla^2f(\mathbf{x})\geq 0$.\\
Moreover for $\mathbf{x}\neq \mathbf{y}$, $f$ is strictly convex $\Leftrightarrow $ $\nabla f(\mathbf{x})^\top(\mathbf{y}-\mathbf{x})< f(\mathbf{y})-f(\mathbf{x})$. Additionally, $\nabla^2f(\mathbf{x})> 0$ implies strict convexity of $f$. Furthermore, if $f(x): \mathbb{R}\rightarrow \mathbb{R}$ is continuously differentiable and  strictly convex, then $\frac{\partial f}{\partial x}(x)$ is a strictly increasing function in $x$.
\begin{definition} (Saddle point definition)
The saddle point $(\mathbf{x}^*,\mathbf{y}^*)\in {\cal X}\times {\cal Y}$, where ${\cal X} \subseteq \mathbb{R}^n$ and ${\cal Y} \subseteq \mathbb{R}^m$, of function $f(\mathbf{x},\mathbf{y})$ is a point on which 
$f(\mathbf{x}^*,\mathbf{y})\leq f(\mathbf{x}^*,\mathbf{y}^*) \leq f(\mathbf{x},\mathbf{y}^*),$
for all $\mathbf{x} \in {\cal X}$ and $\mathbf{y} \in {\cal Y}$.
\end{definition}

We denote ${(\mathbf{x}^*,\boldsymbol{\mu}^*,\boldsymbol{\lambda}^*)}$, where $\mathbf{x}^*=[x_1^*,...,x_n^*]^\top \in \mathbb{R}^n$, $\boldsymbol{\mu}^*=[\mu_1^*,...,\mu_l^*]^\top\in \mathbb{R}^l$,  $\boldsymbol{\lambda}^*=[\lambda_{11}^*,...,\lambda_{ij_i}^*,...,\lambda_{nm_n}^*]^\top \in \mathbb{R}^{\sum_{i=1}^n m_i}_{\geq 0}$, as a saddle point of $\mathfrak{L}(\mathbf{x},{\boldsymbol{\mu}},{\boldsymbol{\lambda}})$. The saddle point satisfies $\mathfrak{L}(\mathbf{x}^*,\boldsymbol{\mu},\boldsymbol{\lambda})\leq \mathfrak{L}(\mathbf{x}^*,\boldsymbol{\mu}^*,\boldsymbol{\lambda}^*)\leq \mathfrak{L}(\mathbf{x},\boldsymbol{\mu}^*,\boldsymbol{\lambda}^*)$ for all $(\mathbf{x},\boldsymbol{\mu},\boldsymbol{\lambda}) \in \mathbb{R}^n\times \mathbb{R}^l \times \mathbb{R}^{\sum_{i=1}^n m_i}_{\geq 0}$. The DCOP outlined in (\ref{GeneralProblem}) is said to be satisfied the Slater condition qualification if there exists some feasible primal solution $\mathbf{x}^*=(x_1^*,...,x_n^*)\in \mathbb{R}^n$ at which $g_{ij_i}(x_i^*)<0$ and $h_e(\mathbf{x}^*)=0$, $i \in {\cal V}$, $j_i \in {\cal G}_i$, and $ e \in {\cal H}$. 
We now state the following two theorems \cite{Durr2011}-\cite{Burke2008} which are used in our paper.
\begin{theorem} \label{theorem1}
Let $f_i$, for all $i \in {\cal V}$, be convex. Let $\mathbf{x}^*\in \mathbb{R}^n$. If there exist $\boldsymbol{\mu}^* \in \mathbb{R}^{l}$ and $\boldsymbol{\lambda}^* \in \mathbb{R}^{\sum_{i=1}^n m_i}_{\geq 0}$ such that $(\mathbf{x}^*,\boldsymbol{\mu}^*,\boldsymbol{\lambda}^*)$ is a saddle point for the Lagrangian $\mathfrak{L}(\mathbf{x},\boldsymbol{\mu},\boldsymbol{\lambda})$ in (\ref{Lagrangian}), then $\mathbf{x}^*$ solves (\ref{GeneralProblem}). Conversely, if $\mathbf{x}^*$ is a solution to (\ref{GeneralProblem}) at which the Slater condition qualification is satisfied, then there exist $\boldsymbol{\mu}^* \in \mathbb{R}^{l}$ and $\boldsymbol{\lambda}^* \in \mathbb{R}^{\sum_{i=1}^n m_i}_{\geq 0}$  such that $(\mathbf{x}^*,\boldsymbol{\mu}^*,\boldsymbol{\lambda}^*)$ is a saddle point for the Lagrangian $\mathfrak{L}(\mathbf{x},\boldsymbol{\mu},\boldsymbol{\lambda})$.
\end{theorem}
\begin{theorem} \label{theorem2} Let $f_i$, for all $i \in {\cal V}$, be convex. Then, $(\mathbf{x}^*,\boldsymbol{\mu}^*,\boldsymbol{\lambda}^*) \in \mathbb{R}^n\times \mathbb{R}^l \times \mathbb{R}^{\sum_{i=1}^n m_i}_{\geq 0}$ is a saddle point of $\mathfrak{L}$ in (\ref{Lagrangian}) if and only if the following conditions are satisfied for $i \in {\cal V}$, $e \in {\cal H}$, and $j_i \in {\cal G}_i$:
\begin{subequations}\label{eq:KKT}
    \begin{align}
    &\nabla f(\mathbf{x}^*)+ \sum_{e=1}^l \mu_e^* \mathbf{a}_e^h+{[\sum_{j_i=1}^{m_i} {\lambda}_{ij_i}^* {a}_{ij_i}^g]_{vec}} =\mathbf{0}, \label{eq:KKT1_1}\\   
       & \sum_{i=1}^n(a_{ie}^hx_i^*+b_{ie}^h)=0,\textrm{ } \lambda^*_{ij_i}\geq 0, \label{eq:KKT1_2}\\
       & a_{ij_i}^g x_i^* + b_{ij_i}^g \leq 0, \label{eq:KKT1_3}\\
       &\lambda_{ij_i}^* ( a_{ij_i}^gx_i^*+b_{ij_i}^g)=0, \label{eq:KKT1_4}     
    \end{align}
\end{subequations}
where $\mathbf{a}_e^h=[a_{1e}^h,...,a_{ne}^h]^\top $ and ${[\sum_{j_i=1}^{m_i} {\lambda}_{ij_i}^* {a}_{ij_i}^g]_{vec}}=[\sum_{j_1=1}^{m_1} {\lambda}_{1j_1}^* {a}_{1j_1}^g,...,\sum_{j_n=1}^{m_n} {\lambda}_{nj_n}^* {a}_{nj_n}^g]^\top$.\end{theorem}

The conditions (\ref{eq:KKT1_1})-(\ref{eq:KKT1_4}) are referred to  Karush-Kuhn-Tucker (KKT) optimality conditions \cite{Durr2012}, \cite{Boyd2004}. Moreover, condition (\ref{eq:KKT1_4}) is known as complementary slackness in the sense of $ \lambda^*_{ij_i} \neq 0 \Longrightarrow ( a_{ij_i}^gx_i^*+b_{ij_i}^g)=0$ and $( a_{ij_i}^gx_i^*+b_{ij_i}^g)\neq 0 \Longrightarrow \lambda_{ij_i}^*=0$.

We continue by stating some assumptions in this work.
\begin{assumption} \label{assumption1}
The objective cost function $f_i$, for all $i\in {\cal V}$, is strictly convex and continuously differentiable.
\end{assumption}

Let ${\cal S}=\{ \mathbf{x}\in \mathbb{R}^n| g_{ij_i}(x_i) \leq 0, i\in {\cal V},j_i \in {\cal G}_i, h_e(\mathbf{x})=0, e \in {\cal H} \}$ be the feasible set and ${\cal S}^*$ be the set of solutions for the problem  (\ref{GeneralProblem}). 
\begin{assumption} \label{assumption2} The set ${\cal S}$ is nonempty and there exists $\mathbf x \in {\cal S}$ such that the Slater's condition is satisfied.
\end{assumption}

Let $\boldsymbol{\psi}_e^h=[\frac{\partial h_e}{\partial x_1}(\mathbf{x}),...,\frac{\partial h_e}{\partial x_n}(\mathbf{x})]^\top \in \mathbb{R}^n, \textrm{ } e \in {\cal H}$, and $\boldsymbol{\psi}_{ij_i}^g=[\frac{\partial g_{ij_i}}{\partial x_1}({x_i}),...,\frac{\partial g_{ij_i}}{\partial x_n}({x_i})]^\top \in \mathbb{R}^n, \textrm{ } i \in {\cal V}$ and $j_i \in {\cal G}_i$. For the sake of convenience, we define the sets ${\cal T}^E=\{ij_i|i \in {\cal V}, j_i \in {\cal G}_i, \lambda_{ij_i}^*=0\}$ and ${\cal T}^I=\{ij_i|i \in {\cal V}, j_i \in {\cal G}_i, \lambda_{ij_i}^* \neq 0\}$. Let $\mathbf{\Psi}=[\boldsymbol{\psi}_1^h,...,\boldsymbol{\psi}_l^h,\underbrace{...,\boldsymbol{\psi}_{ij_i}^g,...}_{ij_i \in {\cal T}^I}]$.

\begin{assumption} \label{assumption3} The matrix $\boldsymbol{\Psi}$ has a full column rank. 
\end{assumption}

As shown in later that \textit{Assumption~\ref{assumption3}} is required  to guarantee the uniqueness of saddle point of $\mathfrak{L}$ in (\ref{Lagrangian}).
\begin{assumption} \label{assumption4} 
The communication graph $\mathcal{G}$ is undirected and connected.
\end{assumption}
\section{Main Results}
\subsection{Distributed Solution}
To tackle the problem (\ref{GeneralProblem}) in a fully distributed and privacy-guaranteed manner, we propose a solution described in (\ref{eq:boundaryLayerX})-(\ref{eq:reducedSystem}) which operate simultaneously. We would like to estimate the average $\frac{1}{n}\sum_{i=1}^n(a^h_{ie}x_i+b_{ie}^h)$, for each $i\in {\cal V}$ and $e\in {\cal H}$, by the following dynamical equations
\begin{subequations}\label{eq:boundaryLayerX}
    \begin{align}
        \dot{\xi}_{ie}^h=&-\xi_{ie}^h-\sum_{j \in {\cal N}_i}(\xi_{ie}^h-\xi_{je}^h)-\sum_{j \in {\cal N}_i}(\zeta_{ie}^h-\zeta_{je}^h) \nonumber \\
        &  +(a^h_{ie}x_i+b_{ie}^h), \\
        \dot{\zeta}_{ie}^h=&\sum_{j \in {\cal N}_i}(\xi_{ie}^h-\xi_{je}^h),
    \end{align}
\end{subequations}
where $\xi_{ie}^h$ and $\zeta_{ie}^h \in \mathbb{R}$. We also estimate the Lagrange multiplier $\mu_e$, for each $i\in {\cal V}$ and $e\in {\cal H}$, by
\begin{subequations} \label{eq:boundaryLayerMu}
    \begin{align}
        \dot{\xi}_{ie}^\mu=&-\xi_{ie}^\mu-\sum_{j \in {\cal N}_i}(\xi_{ie}^\mu-\xi_{je}^\mu) \nonumber \\
        &-\sum_{j \in {\cal N}_i}(\zeta_{ie}^\mu-\zeta_{je}^\mu)+\mu_{ie}, \\
        \dot{\zeta}_{ie}^\mu=&\sum_{j \in {\cal N}_i}(\xi_{ie}^\mu-\xi_{je}^\mu),
    \end{align}
\end{subequations}
where $\xi_{ie}^\mu, \zeta_{ie}^\mu$, and $\mu_{ie} \in \mathbb{R}$.  
The following is called slow dynamics aiming to seek the optimal value for $x_i$, $i \in {\cal V}$:
\begin{subequations} \label{eq:reducedSystem}
    \begin{align}
        \dot{x}_{i}=&-\epsilon k_i^x\Big( \frac{\partial f_i}{\partial x_i}(x_i)+\sum_{e=1}^l \xi_{ie}^\mu \frac{\partial h_e}{\partial x_i}(\mathbf{x}) \nonumber \\
        &\qquad \textrm{ } \textrm{ }+\sum_{j_i=1}^{m_i}\lambda_{ij_i}\frac{\partial g_{ij_i}}{\partial x_i}(x_i)\Big), \label{eq:reducedSystem1}\\
        \dot{\mu}_{ie}=&\epsilon k_{ie}^\mu \Big( \xi_{ie}^h-\sum_{j \in {\cal N}_i}(\mu_{ie}-\mu_{je})\Big), \textrm{ } e \in {\cal H}, \label{eq:reducedSystem2}\\
        \dot{\lambda}_{ij_i}=&\epsilon k_{ij_i}^\lambda \lambda_{ij_i} g_{ij_i}(x_i), \textrm{ }j_i \in {\cal G}_i, \label{eq:reducedSystem3}
  \end{align}
\end{subequations}
where $k_i^x,k_{ie}^{\mu}$, $k_{ie}^{\lambda}\in \mathbb{R}_{>0}$ and $\epsilon$ is a small real positive number.\\
\begin{remark} \label{remark1}
The equations (\ref{eq:boundaryLayerX}) and (\ref{eq:boundaryLayerMu}), which utilize the average dynamic consensus algorithm mentioned in Section II, are considered as fast dynamics, while (\ref{eq:reducedSystem}) is slow dynamics with sufficiently small $\epsilon \in \mathbb{R}_{>0}$. As can be seen in (\ref{eq:boundaryLayerX})-(\ref{eq:reducedSystem}), agent $i$ needs only its local information and its neighbors' information; so the proposed algorithm is fully distributed. Furthermore, the strategy also guarantees the privacy; each agent knows only the average estimate during the information exchange. Since $\mu_e$ is global information, the subsystem (\ref{eq:reducedSystem2}) aims to reach a consensus value $\mu_{ie}=\mu_{je}$, for all $i,j \in {\cal V}$. Additionally, it is stated in \cite{Durr2011}-\cite{Durr2012} that if $\lambda_{ij_i}$ has positive initialization, then it stays non-negative.
\end{remark}
\subsection{Convergence Analysis}
\begin{lemma}\label{lemma3}
Let $\mathfrak{L}(\mathbf{x},\boldsymbol{\mu}, \boldsymbol{\lambda})$ be strictly convex in $\mathbf{x}$ and suppose it possesses at least one saddle point $(\mathbf{x}^*,\boldsymbol{\mu}^*, \boldsymbol{\lambda}^*)$. Then the component $\mathbf{x}^*$ of every saddle point $(\mathbf{x}^*,\boldsymbol{\mu}^*, \boldsymbol{\lambda}^*)$ is unique.
\end{lemma}

The proof of \textit{Lemma~\ref{lemma3}} is similar to that of \textit{Lemma~1} in \cite{Durr2012}; thus, it is omitted.
\begin{lemma}\label{lemma4} Let \textit{Assumptions~\ref{assumption1}-\ref{assumption2}} be satisfied. Then, there is a unique solution for the optimization problem (\ref{GeneralProblem}).
\end{lemma}
\begin{proof}
Since $f_i(x_i)$ is strictly convex, so is $f(\mathbf{x})$. Moreover, $g_{ij_i}(x_i)$  is convex. Thus, ${\cal S}^*$ is nonempty, compact, and convex \cite{Nedic2008}. Under \textit{Assumption~\ref{assumption2}}, Slater's condition is satisfied. Hence, according to \textit{Theorem~\ref{theorem1}}, for each $\mathbf{x}^*\in {\cal S}^*$ there exist  $\boldsymbol{\mu}^*\in \mathbb{R}^{l}$ and $\boldsymbol{\lambda}^*\in \mathbb{R}^{\sum_{i=1}^n m_i}_{\geq 0}$ such that $(\mathbf{x}^*,\boldsymbol{\mu}^*,\boldsymbol{\lambda}^*)$ is a saddle point for $\mathfrak{L}$. Since $f(\mathbf{x})$ is strictly convex, so is $\mathfrak{L}$; thus, the component $\mathbf{x}^*$ of every saddle point $(\mathbf{x}^*,\boldsymbol{\mu}^*,\boldsymbol{\lambda}^*)$ is unique, according to \textit{Lemma~\ref{lemma3}}. This concludes our proof.
\end{proof}

\begin{lemma}\label{lemma5} Let \textit{Assumptions~\ref{assumption1}-\ref{assumption3}} be satisfied. Then, there exists a unique saddle of $\mathfrak{L}$ defined in ({\ref{Lagrangian}}).
\end{lemma}
\begin{proof} According to  \textit{Lemma~\ref{lemma4}}, there exists a unique optimal solution $\mathbf{x}^*$ of (\ref{GeneralProblem}) and there also exists a saddle point $(\mathbf{x}^*,\boldsymbol{\mu}^*,\boldsymbol{\lambda}^*)$ for $\mathfrak{L}$. We then prove this saddle point is unique. 
The KKT condition (\ref{eq:KKT1_1}) can be rewritten as
\begin{equation*}
\mathbf{\Psi}[(\boldsymbol{\mu}^*)^\top \quad (\boldsymbol{\lambda}^{*,I})^\top ]^\top=-\nabla f(\mathbf{x}^*),  
\end{equation*}
where $\boldsymbol{\lambda}^{*,I}=[...,\lambda_{ij_i}^*,...]^\top$ for $ij_i \in {\cal T}^I$. 
 Furthermore, since $f_i$ is strictly convex, $\frac{\partial f_i}{\partial x_i}(\mathit{x}_i)$ is an increasing function; thus, there is a unique $\nabla f(\mathbf{x}^*)$ for the unique $\mathbf{x}^*$.
As a result, under \textit{Assumption~\ref{assumption3}} there exists unique $\boldsymbol{\mu}^*$ and $\boldsymbol{\lambda}^*$ corresponding to  $\mathbf{x}^*$; hence, the saddle point $(\mathbf{x}^*,\boldsymbol{\mu}^*,\boldsymbol{\lambda}^*)$ is unique. This completes our proof.\end{proof}

We now go further to investigate the convergence of our proposed strategy to the unique solution of (\ref{GeneralProblem}). Let $\epsilon$ tends to 0 and \textit{Assumptions~\ref{assumption4}} be satisfied, then $\xi_{ie}^\mu$ and  $\xi_{ie}^h$ become instantaneous \cite{Tan2006} as discussed in Section II; that is $\xi_{ie}^h$ in (\ref{eq:boundaryLayerX}) converges to $\frac{1}{n}\sum_{i=1}^n(a_{ie}^hx_i+b_{ie}^h)$ and $\xi_{ie}^\mu$ in (\ref{eq:boundaryLayerMu}) converges to $\frac{1}{n}\sum_{i=1}^n\mu_{ie}$. Hence,  (\ref{eq:reducedSystem}) can take the following form in time $\tau=\epsilon t$ scale when $\epsilon$ tends to 0:
\begin{subequations}\label{eq:ReducedModelS}
    \begin{align}
       \frac{d x_i }{d \tau }=&- k_i^x\Big( \frac{\partial f_i}{\partial x_i}(x_i)+\sum_{e=1}^l \overline{\mu}_{e} a_{ie}^h +\sum_{j_i=1}^{m_i}\lambda_{ij_i}a^g_{ij_i}\Big), \label{eq:ReducedModelS1}\\
       \frac{d \mu_{ie} }{d \tau } =& k_{ie}^\mu \Big( \overline{h}_e-\sum_{j \in {\cal N}_i}(\mu_{ie}-\mu_{je})\Big), \textrm{ } e \in {\cal H},\label{eq:ReducedModelS2}\\
        \frac{d \lambda_{ij_i} }{d \tau }=& k_{ij_i}^\lambda \lambda_{ij_i} g_{ij_i}(x_i), \textrm{ }j_i \in {\cal G}_i.\label{eq:ReducedModelS3}
    \end{align}
\end{subequations}
where $\overline{\mu}_e=\frac{1}{n}\sum_{i=1}^n\mu_{ie}$ and $\overline{h}_e=\frac{1}{n}\sum_{i=1}^n(a_{ie}^hx_i+b_{ie}^h)$.
\begin{lemma}\label{lemma6} Let \textit{Assumptions~\ref{assumption1}-\ref{assumption4}} be satisfied. Additionally, $\lambda_{ij_i}$ has positive initialization  for all $i \in {\cal V}$ and $j_i \in {\cal G}_i$. The reduced model (\ref{eq:ReducedModelS}) is globally asymptotically stable.
\end{lemma}
\begin{proof}
We introduce the coordinate transformation $\hat{x}_i=x_i - x_i^*$, $\hat{\mu}_{ie}=\mu_{ie} - \mu_{e}^*$ and $\hat{\lambda}_{ij_i}=\lambda_{ij_i} - \lambda_{ij_i}^*$.
Let $\overline{\hat{\mu}}_e=\frac{1}{n}\sum_{i=1}^n \hat{\mu}_{ie}$ and $\overline{\hat{h}}_e=\frac{1}{n}\sum_{i=1}^n a_{ie}^h \hat{x}_i$. It is worth mentioning that $\frac{1}{n}\sum_{i=1}^n(a_{ie}^hx_i^*+b_{ie}^h)=0$ and we also have (\ref{eq:KKT1_1}). Then, (\ref{eq:ReducedModelS}) can be rewritten as
\begin{subequations}\label{eq:transformReducedModelS}
    \begin{align}
       \frac{d \hat{x}_i }{d \tau }=&- k_i^x\Big( \frac{\partial f_i}{\partial x_i}(\hat{x}_i+x_i^*)-\frac{\partial f_i}{\partial x_i}(x_i^*) \nonumber \\
       &\qquad \textrm{ }+\sum_{e=1}^l \overline{\hat{\mu}}_{e} a_{ie}^h +\sum_{j_i=1}^{m_i}\hat{\lambda}_{ij_i}a^g_{ij_i}\Big),\\
       \frac{d \hat{\mu}_{ie} }{d \tau } =& k_{ie}^\mu \Big( \overline{\hat{h}}_e-\sum_{j \in {\cal N}_i}(\hat{\mu}_{ie}-\hat{\mu}_{je})\Big), \textrm{ } e \in {\cal H},\\        
        \frac{d \hat{\lambda}_{ij_i} }{d \tau }=& k_{ij_i}^\lambda (\hat{\lambda}_{ij_i}+\lambda_{ij_i}^*)(a_{ij_i}^g \hat{x}_i+b_{ij_i}^g+a_{ij_i}^g x_i^*),
         \textrm{ }j_i \in {\cal G}_i.
    \end{align}
\end{subequations}
Let us denote $\hat{\mathbf{x}}=[\hat{x}_1,...,\hat{x}_n]^\top\in \mathbb{R}^n$, $\hat{\boldsymbol{\mu}}_e=[\hat{\mu}_{1e},...,\hat{\mu}_{ne}]^\top \in \mathbb{R}^n $, $e \in {\cal H}$, $\tilde{\boldsymbol{\mu}}=[\hat{\boldsymbol{\mu}}_1,...,\hat{\boldsymbol{\mu}}_l]^\top \in \mathbb{R}^{nl}$ and  $\hat{\boldsymbol{\lambda}}=[\hat{\lambda}_{11},...,\hat{\lambda}_{ij_i},...,\hat{\lambda}_{nm_n}]^\top \in \mathbb{R}^{\sum_{i=1}^n m_i}$. For the sake of presentation, we also denote $\hat{\mathbf{w}}=[\hat{\mathbf{x}}^\top,\tilde{\boldsymbol{\mu}}^\top,\hat{\boldsymbol\lambda}^\top]^\top$. Inspired by the Lyapunov function proposed in \cite{Durr2011}-\cite{Durr2012}, we consider the following Lyapunov function:
\begin{equation*}
\begin{split}
V(\hat{\mathbf{w}})=&\sum_{i=1}^n\frac{1}{2k_i^x}\hat{x}_i^2+\sum_{i=1}^n\sum_{e=1}^l \frac{1}{2k_{ie}^\mu} \hat{\mu}_{ie}^2
+\sum_{i=1}^n \sum_{j_i=1,\lambda_{ij_i}^*= 0}^{m_i}\frac{\hat{\lambda}_{ij_i}}{k_{ij_i}^\lambda} \\
&+\sum_{i=1}^n \sum_{j_i=1,\lambda_{ij_i}^*\neq  0}^{m_i} \frac{1}{k_{ij_i}^\lambda} \Big( \hat{\lambda}_{ij_i}-\lambda_{ij_i}^*\textrm{ln}(\frac{\hat{\lambda}_{ij_i}+\lambda_{ij_i}^*}{\lambda_{ij_i}^*})\Big).
\end{split}
\end{equation*}
It is straightforward to see that $\sum_{i=1}^n\frac{1}{2k_i^x}\hat{x}_i^2\geq 0$ and $\sum_{i=1}^n\sum_{e=1}^l \frac{1}{2k_{ie}^\mu} \hat{\mu}_{ie}^2 \geq 0$.  
Note that for $ij_i \in \mathcal{T}^E$, we have $\hat{\lambda}_{ij_i}=\lambda_{ij_i} \geq 0$. For $ij_i \in \mathcal{T}^I$, let us denote $\theta_{ij_i}=\frac{\hat{\lambda}_{ij_i}}{\lambda_{ij_i}^*}$. Then the sum of  the third and the fourth term of $V(\hat{\mathbf{w}})$ can be rewritten as $\sum_{i=1}^n\sum_{ij_i \in \mathcal{T}^E}\frac{\hat{\lambda}_{ij_i}}{k_{ij_i}^{\lambda}}+\sum_{i=1}^n\sum_{ij_i \in \mathcal{T}^I} \frac{\lambda_{ij_i}^*}{k_{ij_i}^\lambda} (\theta_{ij_i}-\textrm{ln}(1+\theta_{ij_i}) )$. The function $f^{\theta}(\theta_{ij_i})=\theta_{ij_i}-\textrm{ln}(1+\theta_{ij_i})\geq 0$ for all $ \theta_{ij_i} \in (-1,+\infty)$, $f^{\theta}(\theta_{ij_i})\rightarrow +\infty$ when $\theta_{ij_i} \rightarrow +\infty$ and $f^{\theta}(\theta_{ij_i})=0$ $\Leftrightarrow$ $\theta_{ij_i}=0$ (or $\hat{\lambda}_{ij_i}=0$) in the case $ij_i \in \mathcal{T}^I$. Therefore, $V(\hat{\mathbf{w}})$ is continuously differentiable, radially unbounded and positive definite on $\mathbb{R}^n \times \mathbb{R}^{nl} \times \mathbb{R}_{>-\boldsymbol{\lambda}^*}^{\sum_{i=1}^n m_i}$.
Taking the derivatives of $V(\hat{\mathbf{w}})$ along the trajectories of (\ref{eq:transformReducedModelS}) and noticing that $\sum_{i=1}^n\Big(\hat{x}_i (\sum_{e=1}^l\overline{\hat{\mu}}_ea_{ie}^h)\Big)=\sum_{i=1}^n\Big( \sum_{e=1}^l (\hat{\mu}_{ie}\overline{\hat{h}}_e)\Big)$, we obtain
\begin{equation*}
\begin{split}
\frac{dV(\hat{\mathbf{w}})}{d\tau}=&-\sum_{i=1}^n\hat{x}_i\Big(\frac{\partial f_i}{\partial x_i}(\hat{x}_i+x_i^*)-\frac{\partial f_i}{\partial x_i}(x_i^*)\Big) \\
                &-\sum_{i=1}^n \Big( \hat{x}_i\sum_{j_i=1}^{m_i} \hat{\lambda}_{ij_i} a_{ij_i}^g \Big)-\sum_{e=1}^l\hat{\boldsymbol{\mu}}_e^\top \mathbf{L} \hat{\boldsymbol{\mu}}_e\\
                &+ \sum_{i=1}^n \sum_{j_i=1}^{m_i}\Big(\frac{1}{k_{ij_i}^\lambda} \frac{d\hat{\lambda}_{ij_i}}{d \tau} - \frac{\lambda_{ij_i}^*}{k_{ij_i}^\lambda}  \frac{\frac{d\hat{\lambda}_{ij_i}}{d \tau} }{\hat{\lambda}_{ij_i}+\lambda_{ij_i}^*} \Big)
\end{split}
\end{equation*}
It is worth mentioning that $\frac{d\hat{\lambda}_{ij_i}}{d \tau}=\frac{d\lambda_{ij_i}}{d \tau}$,  $\lambda_{ij_i}=\hat{\lambda}_{ij_i}+\lambda_{ij_i}^*$, $x_i=\hat{x}_i+x_i^*$ and $\lambda_{ij_i}^*(a_{ij_i}^gx_i^*+b_{ij_i}^g)=0$. Then, we can have
\begin{equation*}
\begin{split}
&\frac{1}{k_{ij_i}^\lambda} \frac{d \hat{\lambda}_{ij_i}}{d \tau} - \frac{\lambda_{ij_i}^*}{k_{ij_i}^\lambda}  \frac{\frac{d \hat{\lambda}_{ij_i}}{d \tau}}{\hat{\lambda}_{ij_i}+\lambda_{ij_i}^*}
=\frac{1}{k_{ij_i}^\lambda} \frac{d \hat{\lambda}_{ij_i}}{d \tau} - \frac{\lambda_{ij_i}^*}{k_{ij_i}^\lambda}  \frac{\frac{d \hat{\lambda}_{ij_i}}{d \tau}}{\lambda_{ij_i}}\\
=&(\hat{\lambda}_{ij_i}+\lambda_{ij_i}^*)(a_{ij_i}^g \hat{x}_i+b_{ij_i}^g+a_{ij_i}^g x_i^*)\\
& - \lambda_{ij_i}^* (a_{ij_i}^g \hat{x}_i+b_{ij_i}^g+a_{ij_i}^g x_i^*)\\
=& \hat{\lambda}_{ij_i} (a_{ij_i}^g \hat{x}_i+b_{ij_i}^g+a_{ij_i}^g x_i^*).
\end{split}
\end{equation*}
It is also worth noticing that $\sum_{i=1}^n \Big( \hat{x}_i\sum_{j_i=1}^{m_i} \hat{\lambda}_{ij_i} a_{ij_i}^g \Big)=\sum_{i=1}^n \sum_{j_i=1}^{m_i}\hat{\lambda}_{ij_i}a_{ij_i}^g\hat{x}_i$. Hence, we have
\begin{equation*}
\begin{split}
\frac{dV(\hat{\mathbf{w}})}{d\tau}=&-\sum_{i=1}^n\hat{x}_i\Big(\frac{\partial f_i}{\partial x_i}(\hat{x}_i+x_i^*)-\frac{\partial f_i}{\partial x_i}(x_i^*)\Big) \\
                &-\sum_{e=1}^l\boldsymbol{\mu}_e^\top \mathbf{L} \boldsymbol{\mu}_e+\sum_{i=1}^n \sum_{j_i=1}^{m_i} \hat{\lambda}_{ij_i} (a_{ij_i}^g x_i^*+b_{ij_i}^g).\\       
\end{split}
\end{equation*}
Since $\frac{\partial f_i}{\partial x_i}(x_i)$ is an increasing function, $-\hat{x}_i\Big( \frac{\partial f_i}{\partial x_i}({\hat x}_i+x_i^*)-\frac{\partial f_i}{\partial x_i}({x}_i^*)\Big)\leq 0$. In addition, it is straightforward to see that $-\hat{\boldsymbol{\mu}}_e^\top\mathbf{L}\hat{\boldsymbol{\mu}}_e \leq 0$. Furthermore, note that $\hat{\lambda}_{ij_i}=\lambda_{ij_i}\geq 0$ for $ij_i \in {\cal T}^E$ and $a_{ij_i}^g x_i^*+b_{ij_i}^g=0$ for $ij_i \in {\cal T}^I$. Considering the KKT condition (\ref{eq:KKT1_3}), we can obtain $\hat{\lambda}_{ij_i} (a_{ij_i}^g x_i^*+b_{ij_i}^g) \leq 0$. Consequently, $\frac{d V}{d \tau}\leq 0$. We can use LaSalle's invariance principle where the largest invariant set $\cal I$ is defined as $\frac{d V}{d \tau} \equiv 0$. The set defined from $\frac{d V}{d \tau} \equiv 0$ is characterized by\\
\indent 1) $\hat{x}_i=0$ for all $i \in {\cal V}$.\\
\indent 2) $\hat{\mu}_{ie}=\hat{\mu}_{je}$ for all $i,j\in \cal V$, $e \in \{ 1,...,l\}$.\\
\indent 3) ${\hat{\lambda}}_{ij_i} (a_{ij_i}^g x_i^*+b_{ij_i}^g)=0$, $i \in {\cal V}$ and $j_i \in {\cal G}_i$.\\
First, $\hat{x}_i=0$ in 1) implies $\frac{d {x}_i}{d \tau}=0$ and $x_i=x_i^*$. From (\ref{eq:ReducedModelS1}), in the original coordinate, we therefore have 
\begin{equation*}
\nabla f(\mathbf{x}^*)+ \sum_{e=1}^l \overline{\mu}_e \mathbf{a}_e^h+{[\sum_{j_i=1}^{m_i} {\lambda}_{ij_i} {a}_{ij_i}^g]_{vec}} =0,
\end{equation*}
which satisfies the condition (\ref{eq:KKT1_1}). Second, from  3) and the complementary slackness condition ${\lambda}_{ij_i}^* (a_{ij_i}^g x_i^*+b_{ij_i}^g)=0$, we can obtain $\lambda_{ij_i}(a_{ij_i}^g x_i^*+b_{ij_i}^g)=0$ satisfying (\ref{eq:KKT1_4}) condition. It can also be seen that $\lambda_{ij_i}\geq 0$ and since $x_i=x_i^*$, $\sum_{i=1}^n(a_{ie}^hx_i+b_{ie}^h)=0$ and $a_{ij_i}^gx_i+b_{ij_i}^g \leq 0$. Hence, the largest invariant set in the original coordinate contains elements that satisfy the KKT conditions (\ref{eq:KKT1_1})-(\ref{eq:KKT1_4}). Then, the largest invariant set in the original coordinate is characterized by\\
\indent 1') $x_i=x_i^*$ for all $i \in {\cal V}$.\\
\indent 2') ${\mu}_{ie}={\mu}_{je}=\overline{\mu}_e=\mu_e^*$ for all $i,j\in \cal V$, $e \in {\cal H}$.\\
\indent 3') ${{\lambda}}_{ij_i} =\lambda_{ij_i}^*$, $i \in {\cal V}$ and $j_i \in {\cal G}_i$.\\
As proved in \textit{Lemma~\ref{lemma5}}, the saddle point $(\mathbf{x}^*,\boldsymbol{\mu}^*,\boldsymbol{\lambda}^*)$ is unique. We therefore can conclude that the reduced model (\ref{eq:ReducedModelS}) is globally asymptotically stable.
\end{proof}

Suppose $(\xi_{ie}^{*,h},\zeta_{ie}^{*,h})$ and $(\xi_{ie}^{*,\mu},\zeta_{ie}^{*,\mu})$ in order are equilibrium points of boundary-layer (\ref{eq:boundaryLayerX}) and (\ref{eq:boundaryLayerMu}) at the equilibrium point of the reduced model. Introducing the change of variables $\hat{\xi}_{ie}^h \triangleq \xi_{ie}^h-\xi_{ie}^{*,h}$, $\hat{\zeta}_{ie}^h \triangleq \zeta_{ie}^h-\zeta_{ie}^{*,h}$, $\hat{\xi}_{ie}^\mu \triangleq \xi_{ie}^\mu-\xi_{ie}^{*,\mu}$ and $\hat{\zeta}_{ie}^\mu \triangleq \zeta_{ie}^\mu-\zeta_{ie}^{*,\mu}$, (\ref{eq:boundaryLayerX})-(\ref{eq:boundaryLayerMu}) can be rewritten for all $i\in {\cal V}$ and $e\in {\cal H}$ as
\begin{subequations}\label{eq:transBoundaryLayerX}
    \begin{align}
        \epsilon \frac{d \hat{\xi}_{ie}^h}{d \tau}=&-\hat{\xi}_{ie}^h-\sum_{j \in {\cal N}_i}(\hat{\xi}_{ie}^h-\hat{\xi}_{je}^h)\nonumber \\
        &-\sum_{j \in {\cal N}_i}(\hat{\zeta}_{ie}^h-\hat{\zeta}_{je}^h)%
        +a^h_{ie}\hat{x}_i, \\
        \epsilon \frac{d \hat{\zeta}_{ie}^h}{d \tau}=&\sum_{j \in {\cal N}_i}(\hat{\xi}_{ie}^h-\hat{\xi}_{je}^h),
    \end{align}
\end{subequations}
\begin{subequations}\label{eq:transBoundaryLayerMu}
    \begin{align}
        \epsilon \frac{d \hat{\xi}_{ie}^\mu}{d \tau}=&-\hat{\xi}_{ie}^\mu-\sum_{j \in {\cal N}_i}(\hat{\xi}_{ie}^\mu-\hat{\xi}_{je}^\mu)  \nonumber\\
        &-\sum_{j \in {\cal N}_i}(\hat{\zeta}_{ie}^\mu-\hat{\zeta}_{je}^\mu) 
        +\hat{\mu}_{ie}, \\
        \epsilon \frac{d \hat{\zeta}_{ie}^\mu}{d \tau}=&\sum_{j \in {\cal N}_i}(\hat{\xi}_{ie}^\mu-\hat{\xi}_{je}^\mu).
    \end{align}
\end{subequations}
For the sake of presentation, we denote some variables as follows. Let $\hat{\boldsymbol{\xi}}^h=[\hat{\xi}_{11}^h,...,\hat{\xi}_{n1}^h,\hat{\xi}_{12}^h,...,\hat{\xi}_{n2}^h,...,\hat{\xi}_{nl}^h]^\top$, $\hat{\boldsymbol{\xi}}^{\mu}=[\hat{\xi}_{11}^{\mu},...,\hat{\xi}_{n1}^{\mu},\hat{\xi}_{12}^{\mu},...,\hat{\xi}_{n2}^{\mu},...,\hat{\xi}_{nl}^{\mu}]^\top$, $\hat{\boldsymbol{\zeta}}^h=[\hat{\zeta}_{11}^h,...,\hat{\zeta}_{n1}^h,\hat{\zeta}_{12}^h,...,$\\
$\hat{\zeta}_{n2}^h,...,\hat{\zeta}_{nl}^h]^\top$, $\hat{\boldsymbol{\zeta}}^{\mu}=[\hat{\zeta}_{11}^{\mu},...,\hat{\zeta}_{n1}^{\mu},\hat{\zeta}_{12}^{\mu},...,\hat{\zeta}_{n2}^{\mu},...,\hat{\zeta}_{nl}^{\mu}]^\top$, 
$\hat{\mathbf{u}}^h=[a_{11}^h\hat{x}_1,...,a_{n1}^h\hat{x}_n,a_{12}^h\hat{x}_1,...,a_{n2}^h\hat{x}_n,$ $...,a_{nl}^h\hat{x}_n]^\top$ and $\hat{\boldsymbol{\mu}}=[\hat{\mu}_{11},...,\hat{\mu}_{n1},\hat{\mu}_{12},...,\hat{\mu}_{n2},...,\hat{\mu}_{nl}]^\top \in \mathbb{R}^{nl}$. 
Denote $\hat{\boldsymbol{\xi}}=[(\hat{\boldsymbol{\xi}}^h)^\top,(\hat{\boldsymbol{\xi}}^{\mu})^\top]^\top$, $\hat{\boldsymbol{\zeta}}=[(\hat{\boldsymbol{\zeta}}^h)^\top, (\hat{\boldsymbol{\zeta}}^\mu)^\top]^\top$ and  $\hat{\mathbf{u}}=[(\hat{\mathbf{u}}^h)^\top,(\hat{\boldsymbol{\mu}})^\top]^\top$ $ \in \mathbb{R}^{2nl}$.
We can have the concatenated form for (\ref{eq:transBoundaryLayerX})-(\ref{eq:transBoundaryLayerMu}) as
\begin{equation}\label{eq:conForm}
\epsilon
\begin{bmatrix}
\frac{d \hat{\boldsymbol{\xi}}}{d \tau}\\
\frac{d \hat{\boldsymbol{\zeta}}}{d \tau}
\end{bmatrix}
=
\begin{bmatrix}
-\mathbf{I}_{2l} \otimes (\mathbf{I}+\mathbf{L})& -\mathbf{I}_{2l} \otimes \mathbf{L}\\
\mathbf{I}_{2l} \otimes  \mathbf{L} &\mathbf{0}
\end{bmatrix} \begin{bmatrix}
\hat{\boldsymbol{\xi}}\\
\hat{\boldsymbol{\zeta}}
\end{bmatrix}
+
\begin{bmatrix}
\hat{\mathbf{u}}\\
\mathbf{0}
\end{bmatrix}.
\end{equation}
Let $\hat{\boldsymbol{\xi}}^s$ and $\hat{\boldsymbol{\zeta}}^s$ be the quasi-steady states of $\hat{\boldsymbol{\xi}}$ and $\hat{\boldsymbol{\zeta}}$, respectively.
Let $\hat{\mathbf{v}}=[\hat{\mathbf{x}}^\top,\hat{\boldsymbol{\mu}}^\top]^\top$ and $\mathbf{y}=[\hat{\boldsymbol{\xi}},\hat{\boldsymbol{\zeta}}]^\top-[\hat{\boldsymbol{\xi}}^s,\hat{\boldsymbol{\zeta}}^s]^\top$. Rewriting (\ref{eq:conForm}) in the following form with time $t$ scale:
\begin{equation}
\frac{d \mathbf{y}}{ d t}=\mathbf{g}(t,\hat{\mathbf{v}},\mathbf{y}+[\hat{\boldsymbol{\xi}}^s,\hat{\boldsymbol{\zeta}}^s]^\top,\epsilon).
\end{equation}

\begin{lemma}\label{lemma7}
Let \textit{Assumptions~\ref{assumption1}-\ref{assumption4}} be satisfied. Additionally, $\lambda_{ij_i}$ has positive initialization  for all $i \in {\cal V}$ and $j_i \in {\cal G}_i$. Then, the origin of the boundary-layer model
\begin{equation}\label{eq:boundaryLayerEp}
\frac{d \mathbf{y}}{ d t}=\mathbf{g}(t,\hat{\mathbf{v}},\mathbf{y}+[\hat{\boldsymbol{\xi}}^s,\hat{\boldsymbol{\zeta}}^s]^\top,0)
\end{equation}
is globally exponentially stable, uniformly in $(t,\hat{\mathbf{v}})$. \end{lemma}
\begin{proof} Let us define an orthogonal matrix $\mathbf{U} \in \mathbb{R}^{n\times n}$ such that 
$\mathbf{U}=[\mathbf{u}_1,...,\mathbf{u}_n]=[\mathbf{U}_1,\mathbf{u}_n]$,
where $\mathbf{u}_i \in \mathbb{R}^n$, $\mathbf{U}_1\in \mathbb{R}^{n \times (n-1)}$, $\mathbf{u}_i^\top\mathbf{u}_j=0$ and $\mathbf{u}_n=\delta \mathbf{1}_n$, where $\delta >0$ is a positive constant, $i, j \in \{1,...,n\}$ and $i \neq j$.
Define $[(\hat{\boldsymbol{\zeta}}_{11})^\top,\hat{\boldsymbol{\zeta}}_{12},...,(\hat{\boldsymbol{\zeta}}_{(2l)1})^\top,\hat{\boldsymbol{\zeta}}_{(2l)2}]^\top=(\mathbf{I}_{2l}\otimes \mathbf{U}^\top) \hat{\boldsymbol{\zeta}}$ where $\hat{\boldsymbol{\zeta}}_{i1}\in \mathbb{R}^{n-1}$ and $\hat{\boldsymbol{\zeta}}_{i2} \in \mathbb{R}^{}$, $i \in \{1,...,2l\}$. Denote $\hat{\boldsymbol{\zeta}}_1=[(\hat{\boldsymbol{\zeta}}_{11})^\top,...,(\hat{\boldsymbol{\zeta}}_{(2l)1})^\top]^\top$ and $\hat{\boldsymbol{\zeta}}_2=[\hat{\boldsymbol{\zeta}}_{12},...,\hat{\boldsymbol{\zeta}}_{(2l)2}]^\top$. Then, (\ref{eq:conForm}) can be rewritten as 
\begin{equation}\label{eq:reducedOrderBound}
\epsilon
\begin{bmatrix}
\frac{d \hat{\boldsymbol{\xi}}}{d \tau}\\
\frac{d \hat{\boldsymbol{\zeta}}_1}{d \tau}
\end{bmatrix}
=
\begin{bmatrix}
-\mathbf{I}_{2l} \otimes (\mathbf{I}+\mathbf{L})& -\mathbf{I}_{2l} \otimes \mathbf{L}\mathbf{U}_1\\
\mathbf{I}_{2l} \otimes \mathbf{U}_1^\top \mathbf{L} &\mathbf{0}
\end{bmatrix} \begin{bmatrix}
\hat{\boldsymbol{\xi}} \\
\hat{\boldsymbol{\zeta}}_1
\end{bmatrix}
+
\begin{bmatrix}
\hat{\mathbf{u}}\\
\mathbf{0}
\end{bmatrix},
\end{equation}
and 
\begin{equation}
\epsilon \frac{d \hat{\boldsymbol{\zeta}}_2}{d \tau}=\mathbf{0}.
\end{equation}
Define $[(\hat{\boldsymbol{\zeta}}_{11}^s)^\top,\hat{\boldsymbol{\zeta}}_{12}^s,...,(\hat{\boldsymbol{\zeta}}_{(2l)1}^s)^\top,\hat{\boldsymbol{\zeta}}_{(2l)2}^s]^\top=(\mathbf{I}_{2l}\otimes \mathbf{U}^\top) \hat{\boldsymbol{\zeta}}^s$ where $\hat{\boldsymbol{\zeta}}_{i1}^s\in \mathbb{R}^{n-1}$ and $\hat{\boldsymbol{\zeta}}_{i2}^s \in \mathbb{R}^{}$, $i \in \{1,...,2l\}$. Denote $\hat{\boldsymbol{\zeta}}_1^s=[(\hat{\boldsymbol{\zeta}}_{11}^s)^\top,...,(\hat{\boldsymbol{\zeta}}_{(2l)1}^s)^\top]^\top$ and $\hat{\boldsymbol{\zeta}}_2^s=[\hat{\boldsymbol{\zeta}}_{12}^s,...,\hat{\boldsymbol{\zeta}}_{(2l)2}^s]^\top$. Since $\hat{\boldsymbol{\xi}}^s$ and $\hat{\boldsymbol{\zeta}}^s$ are quasi-steady states of $\hat{\boldsymbol{\xi}}$ and $\hat{\boldsymbol{\zeta}}$, then we can have
\begin{equation}\label{eq:equilibriumBounadry}
\underbrace{
\begin{bmatrix}
-\mathbf{I}_{2l} \otimes (\mathbf{I}+\mathbf{L})& -\mathbf{I}_{2l} \otimes \mathbf{L}\mathbf{U}_1\\
\mathbf{I}_{2l} \otimes \mathbf{U}_1^\top \mathbf{L} &\mathbf{0}
\end{bmatrix}
}_{\coloneqq \mathbf{A}}
 \begin{bmatrix}
\hat{\boldsymbol{\xi}}^s\\
\hat{\boldsymbol{\zeta}}_1^s
\end{bmatrix}
+
\begin{bmatrix}
\hat{\mathbf{u}}\\
\mathbf{0}
\end{bmatrix}=\mathbf{0}.
\end{equation}
Let $\mathbf{y}_1= \hat{\boldsymbol{\xi}}-\hat{\boldsymbol{\xi}}^s$ and $\mathbf{y}_2=\hat{\boldsymbol{\zeta}}_1-\hat{\boldsymbol{\zeta}}_1^s$.
Then, (\ref{eq:reducedOrderBound}) can have the form as follows:
\begin{equation}
\begin{split}
\epsilon
\begin{bmatrix}
\frac{d \mathbf{y}_1 }{d \tau}\\
\frac{d \mathbf{y}_2}{d \tau}
\end{bmatrix}
=&
\mathbf{A} \begin{bmatrix}
\mathbf{y}_1+\hat{\boldsymbol{\xi}}^s\\
\mathbf{y}_2+\hat{\boldsymbol{\zeta}}_1^s
\end{bmatrix}
+
\begin{bmatrix}
\hat{\mathbf{u}}\\
\mathbf{0}
\end{bmatrix}-\epsilon
\begin{bmatrix}
\frac{d \hat{\boldsymbol{\xi}}^s }{d \tau} \\
\frac{d \hat{\boldsymbol{\zeta}}_1^s}{d \tau}
\end{bmatrix}\\
=&
\mathbf{A} \begin{bmatrix}
\mathbf{y}_1+\hat{\boldsymbol{\xi}}^s\\
\mathbf{y}_2+\hat{\boldsymbol{\zeta}}_1^s
\end{bmatrix}
+
\begin{bmatrix}
\hat{\mathbf{u}}\\
\mathbf{0}
\end{bmatrix}-\epsilon
\begin{bmatrix}
\frac{\partial \hat{\boldsymbol{\xi}}^s }{\partial \hat{\mathbf{v}}} \\
\frac{\partial \hat{\boldsymbol{\zeta}}_1^s}{\partial \hat{\mathbf{v}}}
\end{bmatrix}\frac{d \hat{\mathbf{v}}}{ d\tau}.
\end{split}
\end{equation}
Hence, in $t$-time scale
\begin{equation}
\begin{bmatrix}
\frac{d \mathbf{y}_1 }{d t}\\
\frac{d \mathbf{y}_2}{d t}
\end{bmatrix}
=
\mathbf{A} \begin{bmatrix}
\mathbf{y}_1+\hat{\boldsymbol{\xi}}^s\\
\mathbf{y}_2+\hat{\boldsymbol{\zeta}}_1^s
\end{bmatrix}
+
\begin{bmatrix}
\hat{\mathbf{u}}\\
\mathbf{0}
\end{bmatrix}-\epsilon
\begin{bmatrix}
\frac{\partial \hat{\boldsymbol{\xi}}^s }{\partial \hat{\mathbf{v}}} \\
\frac{\partial \hat{\boldsymbol{\zeta}}_1^s}{\partial \hat{\mathbf{v}}}
\end{bmatrix}\frac{d \hat{\mathbf{v}}}{ d\tau}.
\end{equation}
In time $t$ scale, letting $\epsilon=0$ and due to (\ref{eq:equilibriumBounadry}), we can have 
\begin{equation}
\begin{bmatrix}
\frac{d \mathbf{y}_1 }{d t}\\
\frac{d \mathbf{y}_2}{d t}
\end{bmatrix}
=
\mathbf{A} \begin{bmatrix}
\mathbf{y}_1\\
\mathbf{y}_2
\end{bmatrix}.
\end{equation}
Since the matrix $\mathbf{A}$ is Hurwitz by \textit{Lemma~3} in \cite{Durr2012}, the origin of the boundary-layer model (\ref{eq:boundaryLayerEp}) is globally exponentially stable, uniformly in $(t,\hat{\mathbf{v}})$.
\end{proof}

The following theorem studies the semi-globally practically asymptotically (SPA) stability \cite{Tan2006} of the proposed model. 
\begin{theorem} \label{theorem8}
Let \textit{Assumptions~\ref{assumption1}-\ref{assumption4}} be satisfied. Additionally, $\lambda_{ij_i}$ has positive initialization  for all $i \in {\cal V}$ and $j_i \in {\cal G}_i$. Then, the system described by (\ref{eq:boundaryLayerX}), (\ref{eq:boundaryLayerMu}) and (\ref{eq:reducedSystem}) is SPA stable.
\end{theorem}
\begin{proof} \textit{Lemma~\ref{lemma6}} points out that the reduced model (\ref{eq:ReducedModelS}) is globally asymptotically stable. We also have \textit{Lemma~\ref{lemma7}} which implies that the origin of the boundary-layer model (\ref{eq:boundaryLayerEp}) is globally exponentially stable, uniformly in $(t,\hat{\mathbf{v}})$. Then, we can apply \textit{Lemma~1} in the Appendix of \cite{Tan2006} and can conclude that the system described by (\ref{eq:boundaryLayerX}), (\ref{eq:boundaryLayerMu}) and (\ref{eq:reducedSystem}) is SPA stable. \end{proof}

Roughly speaking, the SPA stability can be interpreted as given a sufficiently large set of initial conditions ${\cal B}_{I}$ for $(\mathbf{x},\mu_{11},...,\mu_{n1},...,\mu_{1l},...,\mu_{nl},\boldsymbol{\lambda})$ and a sufficiently small neighborhood ${\cal B}_{N}$ of $(\mathbf{x}^*,\underbrace{\mu_1^*,...,\mu_1^*}_{\textrm{n elements}},...,\underbrace{\mu_l^*,...,\mu_l^*}_{\textrm{n elements}},\boldsymbol{\lambda}^*)$, it is possible to adjust the  parameter $\epsilon$ so that all solutions starting from the set ${\cal B}_{I}$ eventually converge to ${\cal B}_{N}$.

\section{Application to Energy Network via Numerical Simulation}
We consider a network of 8 generators communicating over an undirected and connected graph as depicted in Fig. 1.
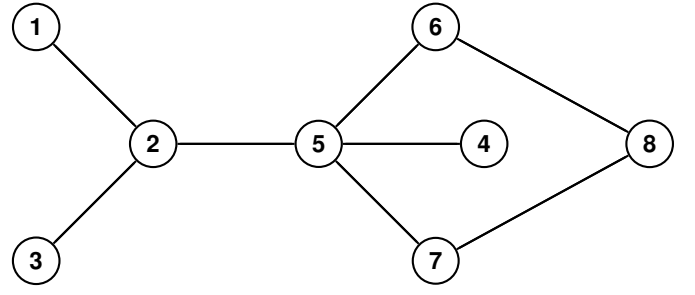
\begin{figure}[!h]
\centering
\begin{tikzpicture}[>=stealth',shorten >=1pt,auto,node distance=2.2cm,
                    thick,main node/.style={circle,draw,font=\sffamily\small\bfseries}]
  
  \node[main node] (2)  {2};
  \node[main node] (1) [above left of=2]{1};
  \node[main node] (3) [below left of=2] {3};      
  \node[main node] (5) [right of=2] {5};
  \node[main node] (7) [below right of=5] {7};
  \node[main node] (4) [right of=5] {4};
  \node[main node] (8) [right of=4] {8};
  \node[main node] (6) [above right of=5] {6};
  \path[thick]  
    (1) edge node  {} (2)
    (2) edge node {} (1)
        edge node  {} (3)   
        edge node  {} (5)   
   (3)   edge node {} (2) 
   (4)   edge node {} (5)
   (5)   edge node {} (2)
         edge node {} (4)
         edge node {} (7)
         edge node {} (6)
    (6)   edge node {} (8)
          edge node {} (5)
    (7)   edge node {} (8)
          edge node {} (5)
    (8)   edge node {} (7)
             edge node {} (6)
        ;            
\end{tikzpicture}
\caption{Communication topology for the simulation.}
\end{figure}
Let $x_i$ be the power generation and $x_i^d$ be the power demand of node $i$. Each generator has a generation cost function $f_i(x_i)=a_ix_i^2+b_ix_i+c_i$, $a_i>0$. It is straightforward to see that the considered quadratic cost function is strictly convex.
We set $\{a_1,...,a_8\}=\{1,3,1,1,1,2,1,1\}$, $\{b_1,...,b_8\}=\{-5,-10,-10,-5,-2,-5,-5,-5\}$, and $c_i=0$ for all $i\in {\cal V}$. We now study two cases which can be encountered in real energy networks.
\subsection{Simulation Case 1}
We consider the case in which the 8 generators cooperatively minimize the total generation cost function in a distributed manner while satisfying supply-demand balance constraint as $\sum_{i=1}^3 x_i=\sum_{i=1}^3 x_i^{d}$ and $\sum_{i=4}^8 x_i=\sum_{i=4}^8 x_i^{d}$.  The power demand at each bus (in p.u.) is given as $\{x_1^d,...,x_8^d\}=\{0.51,0.52,0.53,0.54,0.55,0.56,0.57,0.58 \}$. 
By setting up the scenario, the 8 generators can be divided into 2 clusters; the first cluster has generators 1, 2, and 3, while generators 4, 5, 6, 7, and 8 are in the second one. Each cluster should be able to supply power for it. The simulation results for this case are shown in Fig. 2 and Fig. 3. As can be seen, by applying our algorithm the generators can find the optimal values to minimize the overall cost of generation, while the supply-demand balance for each cluster is satisfied in  Fig. 3.

\begin{figure}[!h]
\centering
\includegraphics[scale=0.65]{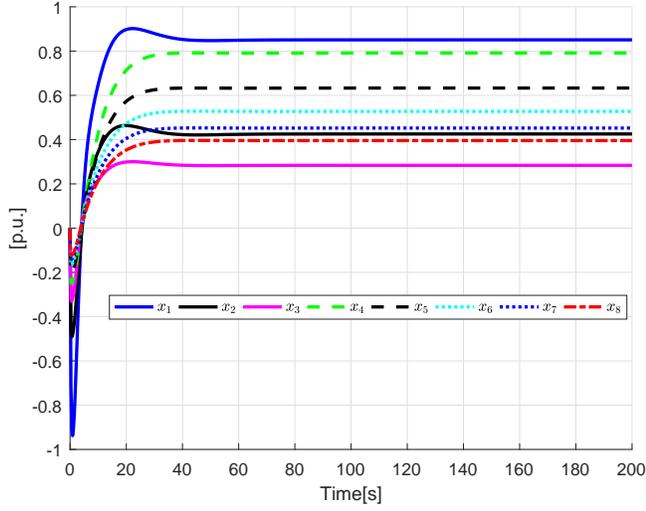}
\caption{Optimal power generation of the simulation case 1.}
\end{figure}

\begin{figure}[!h]
\centering
\includegraphics[scale=0.65]{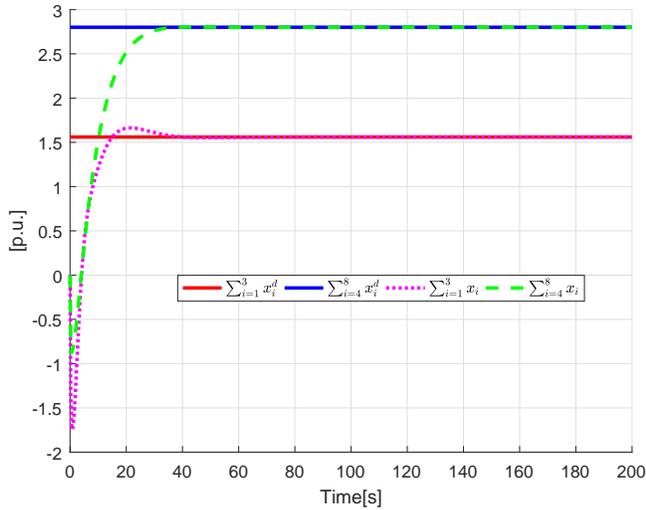}
\caption{Supply-demand balance for cluster of generator 1, 2, and 3 and cluster of generator 4, 5, 6, 7, and 8 of the simulation case 1.}
\end{figure}
\subsection{Simulation Case 2}
We suppose that each generator has a lower and an upper limit capacity, i.e., ${x_i^m} \leq x_i \leq {x_i^M}$. The limit generation capacity (in p.u.) at each generator is given by $0.7\leq x_1\leq 0.9$, $0.3\leq x_2\leq 0.9$, $0.4\leq x_3\leq 0.9$, $0.1\leq x_4\leq 1.0$, $0.1\leq x_5\leq 1.0$, $0.1\leq x_6\leq 1.0$, $0.1\leq x_7\leq 0.9$, and $0.1\leq x_8\leq 0.7$. The power demand is given the same as in the simulation case 1. The energy networked system has supply-demand balance constraint $\sum_{i=1}^8 x_i=\sum_{i=1}^8 x_i^d$ that the simulation results shown in Fig. 5 demonstrate the correctness. Fig. 4 depicts the optimal power generation seeking. As can be seen in Fig. 4, the generators' power generations converge to the optimal values while honoring the limit capacity constraints.
\begin{figure}[!h]
\centering
\includegraphics[scale=0.65]{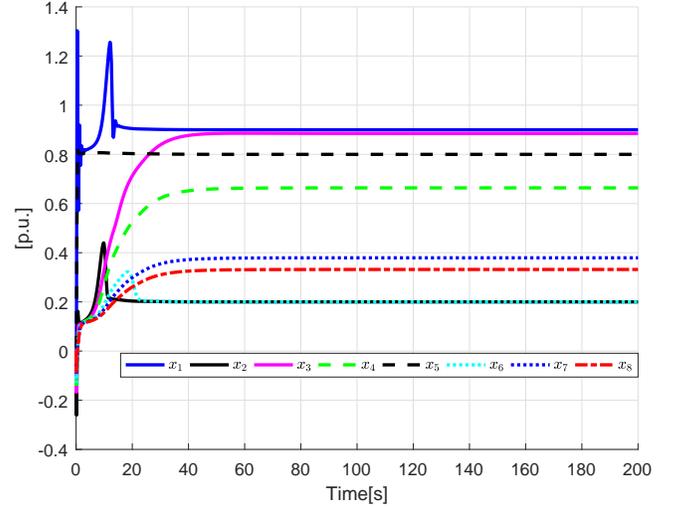}
\caption{Optimal power generation of the simulation case 2.}
\end{figure}

\begin{figure}[!h]
\centering
\includegraphics[scale=0.65]{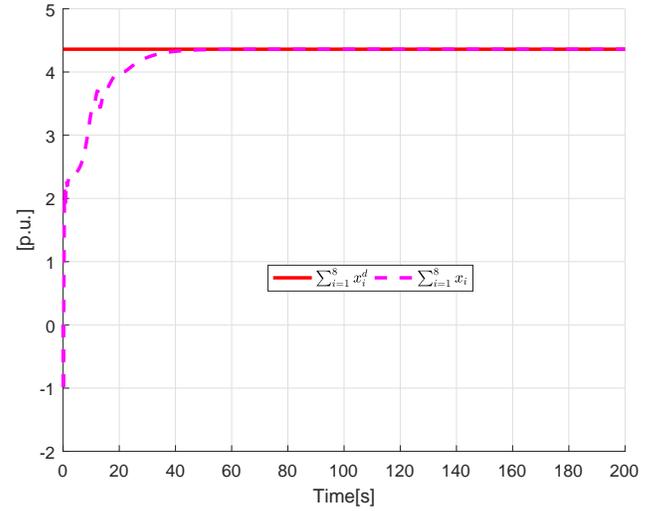}
\caption{Power supply-demand balance of the simulation case 2.}
\end{figure}
\section{Conclusion and Future Work}
We have presented a distributed algorithm for constrained optimization with affine constraints. The fully distributed algorithm is inspired from singular perturbation, dynamic average consensus, and saddle point dynamics methods. The private information of agents in the networked system is guaranteed because we use the dynamic average consensus protocol to estimate average information in boundary-layer systems. The well-developed singular perturbation theory allows us to provide a rigorous analysis on the non-local stability of our proposed algorithm. As demonstrated in the two simulations, the distributed solution can be applied into energy networks. As a future work, we will consider some uncertainties in exchanged information between neighboring nodes.

\section*{Acknowledgment}
The research of this paper has been supported by the National Research Foundation (NRF) of Korea under the grant NRF-2017R1A2B3007034.


\begin{thebibliography}{99}
{\bibitem{Kim2015} B.-Y. Kim, K.-K. Oh, and H.-S. Ahn, \enquote{Coordination and control for energy distribution in distributed grid networks: Theory and application to power dispatch problem,} \textit{Control Engineering Practice}, vol. 43, pp. 21-38, 2015.}

{\bibitem{Ahn2016} H.-S. Ahn, B.-Y. Kim, Y.-H. Lim, B.-H. Lee, and K.-K. Oh, \enquote{Distributed coordination for optimal energy generation and distribution in cyber-physical energy networks,} \textit {IEEE Transactions on Cybernetics}, vol. PP, 2017.}


{\bibitem{Kia2017} S. S. Kia, \enquote{An augmented Lagrangian distributed algorithm for an in-network optimal resource allocation problem,} \textit{2017 American Control Conference}, pp. 3312-3317, 2017.}

{\bibitem{Cherukuri2017} A. Cherukuri, A. D. Dom\'{i}nguez-Garc\'{i}a, and J. Cort\'{e}s, \enquote{Distributed coordination of power generators for a linearized optimal power flow problem,} \textit{2017 American Control Conference}, pp. 3962-3967, 2017.}

{\bibitem{Loia2014} V. Loia and A. Vaccaro, \enquote{Decentralized economic dispatch in smart grids by self-organizing dynamic agents,} \textit{IEEE Transactions on Systems, Man, and Cybernetics: Systems}, vol. 44, no. 4, pp. 397-408, 2014.}

{\bibitem{Bai2016} L. Bai, M. Ye, C. Sun and G. Hu, \enquote{Distributed control for economic dispatch via  saddle point dynamics and consensus algorithms,} \textit{$55^{th}$ IEEE Conference on Decision and Control}, pp. 6934-6939, 2016.}


{\bibitem{Kar2012} S. Kar and G. Hug, \enquote{Distributed robust economic dispatch in power systems: A consensus + innovations approach,} \textit{Power and Energy Society General Meeting}, San Diego, CA, pp. 1-8, 2012.}


{\bibitem{Garcia2012} A. D. Dom\'{i}nguez-Garc\'{i}a, S. T. Cady, and C. N. Hadjicostis, \enquote{Decentralized optimal dispatch of distributed energy resources,} \textit{$51^{st}$ IEEE Conference on Decision and Control}, pp. 3688-3693, 2012.}


{\bibitem{Mudumbai2012}R. Mudumbai, S. Dasgupta, and B. B. Cho, \enquote{Distributed control for optimal economic dispatch of a network of heterogeneous power generators,} \textit{IEEE Transactions on Power Systems}, vol. 27, no. 4, pp. 1750-1760, 2012.}


{\bibitem{Cherukuri2016} A. Cherukuri and J. Cort\'{e}s, \enquote{Initialization-free distributed coordination for economic dispatch under varying loads and generator commitment,} \textit{Automatica}, vol. 74, pp. 183-193, 2016.}


{\bibitem{Yun2017} H. Yun, H. Shim, and H.-S. Ahn,
 \enquote{Initialization-free, join/split-robust, privacy-guaranteed, distributed algorithm for economic dispatch problem,} Submitted to Automatica.}

{\bibitem{Yi2016} P. Yi, Y. Hong, and F. Liu, \enquote{Initialization-free distributed algorithms for optimal resource allocation with feasibility constraints and application to economic dispatch of power systems,} \textit{Automatica}, vol. 74, pp. 259-269, 2016.}









{\bibitem{Freeman2006} R. A. Freeman, P. Yang, and K. M. Lynch, \enquote{Stability and convergence properties of dynamic average consensus estimators,} \textit{$45^{th}$ IEEE  Conference on Decision and Control}, pp. 398-403, 2006.}


{\bibitem{Saber2004} R. Olfati-Saber and R. M. Murray, \enquote{Consensus problems in networks of agents with switching topology and time-delays,} \textit{IEEE Transactions on Automatic Control}, vol. 49, no. 9, pp. 1520-1533, 2004.}


{\bibitem{Kia2013} S. S. Kia, J. Cort\'{e}s, and S. Mart\'{i}nez,\enquote{Singularly perturbed algorithms for dynamic average consensus,} \textit{2013 European Control Conference}, pp. 1758-1763, 2013.}




{\bibitem{Bullo2009} F. Bullo, J. Cort\'{e}s, and S. Mart\'{i}nez, \enquote{Distributed control of robotic
networks,}\textit{American Mathematical Society, Princeton University Press,} 2009}.

{\bibitem{Nedic2010}A. Nedi\'{c}, A. Ozdaglar, and P. A. Parrilo, \enquote{Constrained consensus and optimization in multi-agent networks,} \textit{IEEE Transactions on Automatic Control}, vol. 55, no. 4, pp. 922-938, 2010.}
 
 




 {\bibitem{Khalil2002} H. K. Khalil, \enquote{Nonlinear systems,} \textit{Prentice Hall}, 2002.}

{\bibitem{Durr2011} H.-B. D\"{u}rr and C. Ebenbauer, \enquote{A smooth vector field for saddle point problems,} \textit{$50^{th}$ IEEE Conference on Decision and Control}, pp. 4654-4660, 2011.}

{\bibitem{Durr2012} H.-B. D\"{u}rr, E. Saka, and C. Ebenbauer, \enquote{A smooth vector field for quadratic programming,} \textit{$51^{st}$ IEEE Conference on Decision and Control}, pp. 2515-2520, 2012.}
 
 {\bibitem{Boyd2004} S. Boyd and L. Vandenberghe, \enquote{Convex optimization,} \textit{Cambridge University Press, } 2004.}

{\bibitem{Nedic2008} A. Nedi\'{c}, \enquote{Lecture notes in convex optimization,} Available at \url{http://www.ifp.illinois.edu/~angelia/convex_optimization_lectures.htm}.}

{\bibitem{Burke2008}J. V. Burke, \enquote{Lecture notes in undergraduate nonlinear continuous optimization,} Available at \url{https://sites.math.washington.edu/~burke/crs/516/notes/undergraduate-nco.pdf}.}

 \bibitem{Tan2006}{Y. Tan, D. Ne\v{s}i\'{c}, and I. Mareels, \enquote{On non-local stability properties of extremum seeking control,} \textit{Automatica}, vol. 42, pp. 889-903, 2006.}

 \bibitem{Ye2017}{M. Ye and G. Hu, \enquote{Game design and analysis for price-based demand response: an aggregate game approach,} \textit{IEEE Transactions on Cybernetics}, vol. 47, no. 3, pp. 720-730, 2017.}
 \bibitem{Madan2006}{R. Madan and S. Lall, \enquote{Distributed algorithms for maximum lifetime routing in wireless sensor networks} \textit{IEEE Transactions on Wireless Communications}, vol. 5, pp. 2185-2193, 2006.}
 
  \bibitem{Doan2017}{T. T. Doan and A. Olshevsky, \enquote{Distributed resource allocation on dynamic networks in quadratic time,} \textit{Systems $\&$ Control Letters}, vol. 99, pp. 57-63, 2017.}
  
{\bibitem{Ye2017b} M. Ye and G. Hu, \enquote{Distributed Nash equilibrium seeking by a consensus based approach,} \textit{IEEE Transactions on Automatic Control}, vol. PP, 2017.}

{\bibitem{Sun2017} C. Sun, M. Ye, and G. Hu, \enquote{Distributed time-varying quadratic optimization for multiple agents under undirected graphs,} \textit{IEEE Transactions on Automatic Control}, vol. 62, no. 7, pp. 3687-3694, 2017.}

{\bibitem{Lin2017} P. Lin, W. Ren, and J. A. Farrell, \enquote{Distributed continuous-time optimization: nonuniform gradient gains, finite-time convergence, and convex constraint set,} \textit{IEEE Transactions on Automatic Control}, vol. 62, no. 5, pp. 2239-2253, 2017.}

{\bibitem{Zeng2017} X. Zeng, P. Yi, and Y. Hong, \enquote{Distributed continuous-time algorithm for constrained convex optimizations via nonsmooth analysis approach,} \textit{IEEE Transactions on Automatic Control}, vol. PP, 2017.}

{\bibitem{Kia2015} S. S. Kia, J. Cort\'{e}s, and S. Mart\'{i}nez, \enquote{Distributed convex optimization via continuous-time coordination algorithms with discrete-time communication,} \textit{Automatica}, vol. 55, pp. 254-264, 2015.}

{\bibitem{Li2017} H. Li, S. Liu, Y. C. Soh, and L. Xie, \enquote{Event-triggered communication and data rate constraint for distributed optimization of multiagent systems,} \textit{IEEE Transactions on Systems, Man, and Cybernetics: Systems}, vol. PP, 2017.}

{\bibitem{Yang2017} S. Yang, Q. Liu, and J. Wang, \enquote{Distributed optimization based on a multiagent system in the presence of communication delays,} \textit{IEEE Transactions on Systems, Man, and Cybernetics: Systems}, vol. 47, no. 5, 2017.}


{\bibitem{Wang2017} X. Wang, Y. Hong, P. Yi, H. Ji, and Y. Kang, \enquote{Distributed optimization design of continuous-time multiagent systems with unknown-frequency disturbances,} \textit {IEEE Transactions on Cybernetics}, vol. 47, no. 8, 2017.}

  \bibitem{Yi2015}{P. Yi, Y. Hong, and F. Liu, \enquote{Distributed gradient algorithm for constrained optimization with application to load sharing in power systems,} \textit{Systems $\&$ Control Letters}, vol. 83, pp. 45-52, 2015.}


{\bibitem{Zhu2015}M. Zhu and S. Mart\'{i}nez, \enquote{Distributed optimization-based control of multi-agent networks in complex environments,} \textit{Springer International Publishing,} 2015.}


{\bibitem{SLee2016}S. Lee and A. Nedi\'{c}, \enquote{Asynchronous gossip-based random projection algorithms over networks ,} \textit{IEEE Transactions on Automatic Control}, vol. 61, no. 4, pp. 953-968, 2016.}
 
\end{thebibliography}
\end{document}